\documentclass[a4paper,12pt]{article}
\usepackage{amsmath}
\usepackage{amsthm}
\usepackage{amsfonts}
\usepackage{mathtools}
\usepackage{amssymb}
\usepackage{ulem}
\usepackage{color}
\usepackage{graphicx} 
\usepackage{tikz}
\usepackage{float}
\usetikzlibrary{intersections, calc, arrows.meta}
\usepackage{mleftright}
\usepackage[margin=25mm]{geometry}

\theoremstyle{definition}
\newtheorem{prop}{Proposition}
\newtheorem{defin}[prop]{Definition}

\newtheorem*{rem}{Remark} 

\newtheorem{thm}[prop]{Theorem}
\newtheorem{lem}[prop]{Lemma}

\everymath{\displaystyle}

\title{Lifespan of Classical Solutions to\\ One-Dimensional Quasilinear Wave Equations}
\date{}
\author{Yuusuke Sugiyama and Taro Yamanoi}

\begin{document}

\maketitle

\begin{abstract}
	In this paper, we consider the upper and lower bounds of the lifespan of classical solutions
	of the Cauchy problem for the one-dimensional quasilinear wave equation $u_{tt}-c(u_x)^2u_{xx}=0$ where the derivative of $c(\theta)$ tends to $0$ near the origin.
	In particular, our result shows that the lifespan of the solution extends algebraically depending on the smallness of the initial data.
	Furthermore, we also show that when $c(\theta)$ is flat at the origin ($c'(\theta)$ and any higher order derivatives vanishes at the origin), the lifespan extends exponentially depending on the smallness of the initial data.
	Our proof is based on the method of Lax’s characteristics and Riemann invariants.

\end{abstract}

\section{Introduction}

\subsection{Problem}

In this paper, we consider the following Cauchy problem of the one-dimensional quasilinear wave equation.
\begin{equation} \label{Eq:Main}
	\begin{cases}
		u_{tt} - c(u_x)^2u_{xx} = 0 \hspace{8pt}  (t,x) \in [0,T]\times \mathbb{R}, \\
		u(0,x)=\varepsilon \varphi(x),\,u_t(0,x)=\varepsilon \psi(x) \hspace{8pt}  x \in \mathbb{R}.
	\end{cases}
\end{equation}
Let \(p>2\), and let \(U\subset \mathbb{R}\) be a neighborhood of the origin.
Throughout this paper, we assume that
\begin{equation*}
	c(0)=1 \hspace{10pt} \text{and} \hspace{10pt} c(\theta)>0 \hspace{10pt} \text{for all }\theta \in U.
\end{equation*}
In addition, we assume that either
\begin{equation}\label{Eq:c assump1}
	\begin{aligned}
		 & c\in C^{\lceil p-2\rceil,p-\lceil p-1 \rceil}(U), \\
		 & c'(0)=\cdots=c^{(\lceil p-2\rceil)}(0)=0,
	\end{aligned}
\end{equation}
for some constant \(C>0\), or
\begin{equation}\label{Eq:c assump p=2}
	c\in C^1(U),\qquad c'(0)\ne 0.
\end{equation}
The following inequality derived from the assumption \eqref{Eq:c assump1} will be used frequently in the proof:
\begin{equation*}
	\bigl|c^{(\lceil p-2\rceil)}(\theta)\bigr|
	\le C|\theta|^{p-\lceil p-1\rceil}
	\quad \text{for all } \theta \in U
\end{equation*}
The ceiling function, denoted by $\lceil x \rceil$, is defined as the least integer greater than or equal to a given real number $x$.
We note that the condition \eqref{Eq:c assump p=2} formally corresponds to the case \(p=2\).

\begin{rem}
	Note that the existence and uniqueness of a local-in-time solution to the Cauchy problem (1) is known for $\varphi \in C_b^2(\mathbb{R})$ and $\psi \in C_b^1(\mathbb{R})$\,(Courant-Lax \cite{Courant}, Hartman-Wintner \cite{Philip}, Douglis \cite{Douglis}).
\end{rem}
By the known local-in-time existence and uniqueness theorem for classical solutions, we define the lifespan $T^*$ by
\begin{align*}
	T^* = \sup \biggl\{ T>0 \;\biggm|\; \sup_{t\in [0,T)} \Bigl(
	 & \|u(t)\|_{L^\infty} +\|u_t(t)\|_{L^\infty}+\|u_x(t)\|_{L^\infty} + \|u_{tt}(t)\|_{L^\infty}                                                                          \\
	 & + \|u_{tx}(t)\|_{L^\infty}+\|u_{xx}(t)\|_{L^\infty} \Bigr) < \infty \hspace{10pt} \text{and} \hspace{10pt} \overline{u_x([0,T)\times \mathbb{R})} \subset U\biggr\}.
\end{align*}
Here we state the main results of this paper, which provide upper and lower estimates for the lifespan of the solution.
\begin{thm}[Existence]
	Let $\varphi \in C_b^2(\mathbb{R})$ and $\psi \in C_b^1(\mathbb{R})$.
	Then  there exists a constant $\varepsilon_0>0$ such that
	\begin{align*}
		C\varepsilon^{-(p-1)}\leq & \,T^*  \hspace{8pt}\text{if}\hspace{9pt}0<\varepsilon \leq \varepsilon_0 \hspace{8pt} \text{for the assumption \eqref{Eq:c assump1}},    \\
		C\varepsilon^{-1}\leq     & \,T^*  \hspace{8pt}\text{if}\hspace{9pt}0<\varepsilon \leq \varepsilon_0 \hspace{8pt} \text{for the assumption \eqref{Eq:c assump p=2}},
	\end{align*}
	where $C>0$ is a positive constant independent of $\varepsilon$.
\end{thm}

\begin{thm}[Blowup]
	Let $\varphi \in C_b^2(\mathbb{R})$ and $\psi \in C_b^1(\mathbb{R})$ and further suppose that $\varphi_x(x), \psi(x) \rightarrow 0$ as $|x| \rightarrow \infty$.
	Moreover, we suppose that $c(\theta)$ is a function satisfying \eqref{Eq:c assump1} and there exists a constant $C>0$ such that, for $\theta \in U$,
	\begin{equation}\label{Eq:c assump2}
		C|\theta|^{p-2} \leq c'(\theta),
	\end{equation}
	or
	\begin{equation}\label{Eq:c assump p=2 2}
		c \in C^1(U) , \quad c'(0) > 0.
	\end{equation}
	Furthermore, we assume that there exists a point $x_0 \in \mathbb{R}$ such that one of the following two conditions is satisfied:
	\begin{align*}
		(\mathrm{i})  & \,\,\psi_x(x_0)+\varphi_{xx}(x_0)>0, \\[7pt]
		(\mathrm{ii}) & \,\,\psi_x(x_0)-\varphi_{xx}(x_0)>0.
	\end{align*}
	Then there exists a constant $\varepsilon_0>0$ such that
	\begin{align*}
		T^* \leq   C'\varepsilon^{-(p-1)} & \hspace{8pt}\text{if}\hspace{9pt}0<\varepsilon \leq \varepsilon_0 \hspace{8pt} \text{for the assumptions \eqref{Eq:c assump1} and \eqref{Eq:c assump2}}, \\
		T^*  \leq  C'\varepsilon^{-1}     & \hspace{8pt}\text{if}\hspace{9pt}0<\varepsilon \leq \varepsilon_0 \hspace{8pt} \text{for the assumption \eqref{Eq:c assump p=2 2} }.
	\end{align*}
	where $C'>0$ is a positive constant independent of $\varepsilon$.

	Moreover,
	\begin{equation*}
		\limsup_{t \rightarrow T^*}\{\|u_{tx}(t)\|_{L^\infty}+\|u_{xx}(t)\|_{L^\infty}\}=\infty.
	\end{equation*}
\end{thm}
Examples of functions $c(\theta)$ satisfying the assumptions of Theorems 1 and 2 include the following:
\begin{align*}
	 & c(\theta)=(1+A|\theta|^{p-2}\theta)^{\frac{1}{2}}  \,\,(A>0),               \\
	 & c(\theta)=e^{C|\theta|^{p-2}\theta},                                        \\
	 & c(\theta)=1+A|\theta|^{p-2}\theta+B|\theta|^{q-2}\theta\,\,(A,B>0\,,\,q>p).
\end{align*}

\begin{rem}
	Under the condition $\psi_x(x) \pm c(\varepsilon \varphi_x(x))\varphi_{xx}(x) \leq 0$ for all $x \in \mathbb{R}$ and $c'(0)>0$, the Cauchy problem \eqref{Eq:Main} possesses a time-global solution.
\end{rem}

\subsection{Known results}
By setting $u_x=v$ and $u_t=w$, the equation we treat in this paper can be reduced to a $2\times2$ system of hyperbolic conservation laws
\begin{equation*}
	\begin{cases}
		v_t-w_x=0, \\
		w_t-c(v)^2v_x=0.
	\end{cases}
\end{equation*}
Lax \cite{Lax64} and John \cite{John} investigated more general $2\times 2$ and $n \times n$ systems of hyperbolic conservation laws, respectively.
By applying these results, we can obtain
\begin{equation*}
	C\varepsilon^{-1} \leq T^* \leq C'\varepsilon^{-1}
\end{equation*}
for the equation
\begin{equation*}
	u_{tt}-c(u_x)^2u_{xx}=0
\end{equation*}
for the case $c(u_x)>0\,,\,c'(0) \not = 0$.
Note that, similar to our main theorem, the lower bound holds for general initial data, while the upper bound holds for special initial data.
In the case of the assumption \eqref{Eq:c assump p=2}, our result is identical to that of Lax and John and is not new.
The true novelty of our work lies in the case of the assumption \eqref{Eq:c assump1}.
By Sugiyama \cite{Sugiyama16}, a sufficient condition for equation
\begin{equation*}
	u_{tt}-(c(u)^2u_x)_x=0 \hspace{9pt} \text{with}\hspace{9pt}  c(u)=(1+u)^a,\,a>0
\end{equation*}
to degenerate, that is, to become $c(u) \rightarrow 0$, is given.
For the problem considered in the present paper, this suggests that \(u_x\) can actually leave the domain \(U\) of \(c(\theta)\) in finite time.
Furthermore, Chen, Pan and Zhu \cite{CPZ} established necessary and sufficient conditions for the blowup of the compressible Euler equation
\begin{equation*}
	\begin{cases}
		u_t - v_x = 0, \\
		v_t + p(u)_x=0,
	\end{cases}
\end{equation*}
where $p(u)=Ku^{-\gamma},\,\gamma>1$,which forms a hyperbolic system.
The equations described above all satisfy condition $c'(0)\not=0$. Equations possessing this property are termed genuinely nonlinear and have been the subject of extensive research since their definition in Lax \cite{Lax57}.
It is therefore natural to ask whether the blowup occurs in the case where $c'(0)=0$, which is close to the linear wave equation\, $u_{tt}-c^2u_{xx}=0.$
In the direction of John’s work \cite{John}, Li, Zhou and Kong \cite{LZK} and Hoshiga \cite{Hoshiga} obtained related results for general quasilinear hyperbolic systems.
In \cite{LZK}, Li, Zhou and Kong introduced the notion of weak linear degeneracy, which lies between genuine nonlinearity and linear degeneracy, and obtained upper and lower bounds for the lifespan of solutions.
In terms of the present problem, their finite-order setting corresponds to a smooth integer-order counterpart of assumption \eqref{Eq:c assump1}.
However, their result is formulated for sufficiently smooth coefficients and their initial data are assumed to decay faster than $|x|^{-1}$ at spatial infinity.
Hoshiga \cite{Hoshiga} studied the corresponding critical case for $2\times2$ quasilinear hyperbolic systems.
In the present setting, this corresponds to the limiting case $p=\infty$ in assumption \eqref{Eq:c assump1}, that is, the case where $c'(\theta)$ vanishes at the origin to infinite order.
He obtained upper and lower estimates of exponential type for the lifespan of solutions.
His result also assumes sufficiently fast spatial decay of the initial data (see also the section $5$).
Their results are based on John's method.
On the other hand, our argument is based on the method of Riemann invariants
and characteristics in the sense of Lax. Owing to the special structure of
the second-order equation considered here, the proof is considerably shorter
and more direct.
More precisely, in proving the lower bound for \(T^*\), we estimate \(c'(u_x)\) from above by using Taylor's formula and the inverse function theorem.
In proving the upper bound for \(T^*\), the key point is to use the lower bound for \(T^*\), which allows us to take sufficiently large times and to obtain a lower estimate for \(G(u_x)\) along characteristic curves.
In particular, we can treat wave speeds with non-integer orders of degeneracy, including the case
$c'(\theta)\sim |\theta|^{p-2},\, p>2,$ under Hölder-type regularity assumptions.
Haruyama and Takamura \cite{Haruyama} treated the case
$c(\theta)=\bigl(1+A|\theta|^{p-2}\theta\bigr)^{1/2}$ for $p>1$. By using an integral equation and an iteration argument, they proved the upper bound $T^*\le C\varepsilon^{-(p-1)}$.
They imposed compact support assumptions and a nonnegativity condition on certain integral quantities determined by the initial data, which are more restrictive than those in the present paper.
Also, when $p<2$, local-in-time existence and uniqueness of classical solutions are not known.

\subsection*{Notation}
The family of functions $C^{m,\alpha}(\Omega)$ is the class of m-times differentiable functions whose m-th derivatives are Hölder continuous with exponent $\alpha$ in $\Omega$.
$C_b^m(\mathbb{R})$ is the set of m-times differentiable functions such that the functions themselves and all their derivatives up to order m are bounded and continuous on $\mathbb{R}$.
Also, throughout this paper, $C$ and $C_j$ denote generic positive constants whose values may change from line to line.

\section{Preliminaries}

Let $G(u_x)=\int_{0}^{u_x}c(\theta)d\theta$ \,\,and define the Riemann invariants as follows:

\begin{equation}\label{Eq:Riemann1}
	\begin{cases}
		r(t,x)=u_t(t,x)+G(u_x), \\[10pt]
		s(t,x)=u_t(t,x)-G(u_x).
	\end{cases}
\end{equation}
From \eqref{Eq:Main} and the definitions of \,$G(u_x)$,

\begin{equation}\label{Eq:Riemann2}
	\begin{cases}
		r_t-c(u_x)r_x=0, \\
		s_t+c(u_x)s_x=0.
	\end{cases}
\end{equation}
From the equation in \eqref{Eq:Riemann1}, differentiating the both sides of the equations in \eqref{Eq:Riemann2} with $x$, we have

\begin{equation}\label{Eq:Riemann3}
	\begin{cases}
		(\partial_t-c\partial_x)r_x=\frac{c'}{2c}r_x(r_x-s_x), \\[10pt]
		(\partial_t+c\partial_x)s_x=\frac{c'}{2c}s_x(s_x-r_x).
	\end{cases}
\end{equation}
For $c(u_x)$, the positive and negative characteristics are solutions to
the following differential equations:
\begin{equation*}
	\begin{cases}
		\frac{dx_\pm}{dt}(t)=\pm c(u_x(t,x_\pm(t))), \\[10pt]
		x_\pm(\tau)=y.
	\end{cases}
\end{equation*}
We denote the characteristics through $(\tau,y)$ by $x_\pm(t;\tau,y)$
and, we simply write $x_\pm(t)$ for $x_\pm(t;\tau,y)$ unless there is a risk of confusion. Using the chain rule, from \eqref{Eq:Riemann2}, we obtain:

\begin{equation}\label{Eq:Invariance1}
	\begin{cases}
		\frac{d}{dt}r(t,x_{-}(t))=0, \\[10pt]
		\frac{d}{dt}s(t,x_{+}(t))=0.
	\end{cases}
\end{equation}
Integrating \eqref{Eq:Invariance1} over $[0,t]$, we have
\begin{equation}\label{Eq:Invariance2}
	\begin{cases}
		r(t,x_{-}(t))=r(0,x_{-}(0)), \\[10pt]
		s(t,x_{+}(t))=s(0,x_{+}(0)).
	\end{cases}
\end{equation}
That is, $r$ and $s$ are invariant on $x_{-}(t)$ and $x_{+}(t)$, respectively.
The following conservation laws hold for the Riemann invariants $r,s$.

\begin{prop}
	For $t \in [0,T^*)$,
	\begin{equation}\label{Eq:Conservation}
		\|r(t)\|_{L^\infty}=\|r(0)\|_{L^\infty}, \quad \|s(t)\|_{L^\infty}=\|s(0)\|_{L^\infty}.
	\end{equation}
\end{prop}

\begin{proof}
	This follows easily from \eqref{Eq:Invariance2} and that the maps from $x$ to $x_{\pm}(t;0,x)$ for all $t \in [0,T^*)$ are bijections.
\end{proof}
\vspace{10pt}
Let $F_1=\sqrt{c}r_x$ and $F_2=\sqrt{c}s_x$, then we obtain the following Riccati-type differential equations:

\begin{align}
	 & \quad \frac{d}{dt}F_1(t,x_{-}(t))=\gamma F_1^2,  \label{Eq:Riccati1} \\[3pt]
	 & \quad \frac{d}{dt}F_2(t,x_{+}(t))=\gamma F_2^2   \label{Eq:Riccati2}
\end{align}
where $\gamma=\gamma(u_x)=\frac{c'(u_x)}{2c(u_x)^{\frac{3}{2}}}$.
Integrating \eqref{Eq:Riccati1} and \eqref{Eq:Riccati2} over $[0,t]$, we have
\begin{equation}\label{Eq:integral}
	F_j(t)=F_j(0)+\int_{0}^{t} \gamma(u_x(\tau))F_j(\tau)^2 d\tau \hspace{8pt} (j=1,2).
\end{equation}
Indeed,
\begin{align*}
	\frac{d}{dt}F_1(t,x_{-}(t)) & =\frac{d}{dt}\{\sqrt{c(u_x(t,x_{-}(t)))}r_x(t,x_{-}(t))\}                                               \\
	                            & =\frac{c'}{2\sqrt{c}}(\partial_t-c\partial_x)u_x \cdot r_x + (\partial_t-c\partial_x)r_x \cdot \sqrt{c} \\
	                            & =\frac{c'}{2\sqrt{c}}s_xr_x + \frac{c'}{2\sqrt{c}}r_x(r_x-s_x)                                          \\
	                            & = (\frac{c'}{2c^{\frac{3}{2}}})(\sqrt{c}r_x)^2.
\end{align*}
The same holds for $F_2$.
The calculations above are formal because the third-order derivative of $u$ appears, but the derivation of the integral equation \eqref{Eq:integral} is rigorously justified by Manfrin \cite{Manfrin}.
\section{Proof of Theorem 1}
In what follows, we consider only the case of assumption \eqref{Eq:c assump1}. The case of assumption \eqref{Eq:c assump p=2} can be handled in exactly the same way.
\begin{prop}
	There exists $\varepsilon_0>0$ such that if $0 < \varepsilon \leq \varepsilon_0$, then, for $t \in [0,T^*)$,
	\begin{align}
		\|u_x(t)\|_{L^\infty}      & \leq C_1\varepsilon,  \label{Eq:u_x}                \\
		\|c'(u_x(t))\|_{L^\infty}  & \leq C_1\varepsilon^{p-2}, \label{Eq:c_prime_upper} \\
		\|c(u_x(t))-1\|_{L^\infty} & \leq C_1\varepsilon^{p-1}. \label{Eq:c_minus_1}
	\end{align}
\end{prop}
\begin{proof}
	By definition of Riemann invariant $r$, we have
	\begin{equation*}
		r(0,x)=\varepsilon \psi(x) + G(\varepsilon \varphi_x).
	\end{equation*}
	Noting that $c(0)=1$ and $c(\theta)$ is continuous, we obtain
	\begin{align*}
		|G(\varepsilon \varphi_x)| & =\left|\int_{0}^{\varepsilon \varphi_x}c(\theta)d\theta \, \right| \\
		                           & \leq C\varepsilon\|\varphi_x\|_{L^\infty}
	\end{align*}
	for sufficiently small $\varepsilon>0$.
	Therefore,
	\begin{equation*}
		|r(0,x)| \leq \varepsilon \|\psi\|_{L^\infty}+C\varepsilon\|\varphi_x\|_{L^\infty}.
	\end{equation*}
	Since the same applies to $s$, taking $\varepsilon>0$ sufficiently small, from \eqref{Eq:Conservation} we have
	\begin{align*}
		\|r(t)\|_{L^\infty} & =\|r(0)\|_{L^\infty} \leq C\varepsilon , \\
		\|s(t)\|_{L^\infty} & =\|s(0)\|_{L^\infty} \leq C\varepsilon.
	\end{align*}
	\vspace{10pt}
	Since $G(u_x)=\frac{1}{2}(r-s)$, combined with the above, we get
	\begin{equation}\label{Eq:G_upper}
		\|G(u_x)\|_{L^\infty} \leq C\varepsilon.
	\end{equation}
	Furthermore, since $c(\theta)>0$ for $\theta \in U$,
	\begin{equation*}
		u_x=G^{-1}(G(u_x))
	\end{equation*}
	holds. Moreover, since $G$ is differentiable in a neighborhood of the origin and $G'(0)=1 \neq 0$, and \eqref{Eq:G_upper} holds,
	the inverse function theorem implies that \(G^{-1}\) is differentiable near the origin and that $(G^{-1})'$ is bounded there.
	Thus, by the fundamental theorem of calculus,
	\begin{align*}
		|u_x|=|G^{-1}(G(u_x))| & \leq \left|\int_{0}^{G(u_x)} \left|\frac{dG^{-1}}{d\theta}(\theta)\right|d\theta\right| \\
		                       & \leq C\varepsilon
	\end{align*}
	holds from \eqref{Eq:G_upper}. Therefore, \eqref{Eq:u_x} holds.
	For $2<p\leq3$, \eqref{Eq:c_prime_upper} and \eqref{Eq:c_minus_1} are clear from  the assumption \eqref{Eq:c assump1}.
	For $p>3$, by Taylor expansion of $c'(u_x)$ up to order $\lceil p-3 \rceil$ and the assumption \eqref{Eq:c assump1},
	\begin{align*}
		|c'(u_x)| & =\left|\frac{c^{\lceil p-2\rceil}(\xi u_x)}{\lceil p-3 \rceil !}\right||u_x|^{\lceil p-3 \rceil} \\
		          & \leq C|u_x|^{p-\lceil p-1 \rceil+\lceil p-3 \rceil}                                              \\
		          & = |u_x|^{p-2}
	\end{align*}
	holds. Therefore, \eqref{Eq:c_prime_upper} and \eqref{Eq:c_minus_1} also hold.
\end{proof}

From \eqref{Eq:c_minus_1}, we have
\begin{equation}\label{Eq:c_near_1}
	\frac{1}{2} \leq c(u_x) \leq \frac{3}{2} \hspace{8pt}\text{if} \hspace{8pt} 0<\varepsilon \leq \varepsilon_0.
\end{equation}

\begin{proof}[Proof of Theorem 1]

	Fix $x\in \mathbb{R}$ arbitrarily.
	By \eqref{Eq:Riccati1}, on the characteristic $x_{-}(t)$, we have
	\begin{equation*}
		F_1(t)=F_1(0)+\int_{0}^{t}\gamma(u_x(\tau))F_1(\tau)^2d\tau.
	\end{equation*}
	From \eqref{Eq:c_prime_upper} and \eqref{Eq:c_near_1}, there exists a positive constant $\gamma_1>0$ such that for sufficiently small $\varepsilon>0$,
	\begin{equation*}
		|\gamma(u_x(t))| \leq \gamma_1\varepsilon^{p-2}.
	\end{equation*}
	Furthermore, on the characteristic curve $x_{-}(t)$ passing through $(0,x)$,
	\begin{equation*}
		|F_1(0)| =|\sqrt{c(u_x(0,x))}  (\varepsilon \psi_x(x)  +\varepsilon c(u_x(0,x))\varphi_{xx}(x))| \leq C\varepsilon,
	\end{equation*}
	hence,
	\begin{equation*}
		|F_1(t)|\leq C\varepsilon+\gamma_1\varepsilon^{p-2} \int_{0}^{t}F_1(\tau)^2d\tau.
	\end{equation*}
	Here, let $f(t)$ be the solution to the integral equation
	\begin{equation*}
		f(t)=2C\varepsilon+\gamma_1\varepsilon^{p-2}\int_{0}^{t}f(\tau)^2d\tau.
	\end{equation*}
	Solving this integral equation, we have
	\begin{equation*}
		f(t)=\frac{2C\varepsilon}{1-2C\gamma_1\varepsilon^{p-1}t}
	\end{equation*}
	and we can easily check that \,$|F_1(t)|< f(t)$\, by comparison argument.
	Since a similar estimate holds for $F_2$, we obtain
	\begin{equation*}
		\frac{1}{2C\gamma_1}\varepsilon^{-(p-1)} \leq T^*.
	\end{equation*}
\end{proof}

\section{Proof of Theorem 2}
In what follows, we consider only the case of the assumptions \eqref{Eq:c assump1} and \eqref{Eq:c assump2}. The case of the assumption \eqref{Eq:c assump p=2 2} can be handled in exactly the same way.
Assumptions $(\mathrm{i})$ and $(\mathrm{ii})$ yield $(\mathrm{i})'$ and $(\mathrm{ii})'$, respectively. There exists a point $x_0' \in \mathbb{R}$
\begin{align*}
	(\mathrm{i})'  & \,\,\psi(x_0')+ \varphi_x(x_0')\neq0 \hspace{8pt} \text{and} \hspace{8pt} \psi_x(x_0')+\varphi_{xx}(x_0')>0 \\[7pt]
	(\mathrm{ii})' & \,\,\psi(x_0')- \varphi_x(x_0')\neq0 \hspace{8pt} \text{and} \hspace{8pt} \psi_x(x_0')-\varphi_{xx}(x_0')>0
\end{align*}
We relabel this point $x_0'$ as $x_0$.
\begin{lem}
	For sufficiently small $\varepsilon>0$ and sufficiently large time $t>0$,
	\begin{align}
		\text{Under assumption (i),} \,\,  & c'(u_x) \geq C\varepsilon^{p-2} \hspace{7pt}\text{on $x_{-}(t;0,x_0)$} ,\label{Eq:c_prime_lower1} \\
		\text{Under assumption (ii),} \,\, & c'(u_x) \geq C\varepsilon^{p-2} \hspace{7pt}\text{on $x_{+}(t;0,x_0)$} .\label{Eq:c_prime_lower2}
	\end{align}
\end{lem}

\begin{proof}
	First, from \eqref{Eq:c_near_1} and the definition of $G(u_x)$,
	\begin{equation}\label{Eq:G_to_u}
		\frac{1}{2}|u_x| \leq |G(u_x)| \leq \frac{3}{2}|u_x|
	\end{equation}
	\vspace{2pt}
	holds for sufficiently small $\varepsilon>0$. Since $G(u_x)=\frac{1}{2}(r-s)$, we have
	\begin{equation}\label{Eq:G_expand}
		2G(u_x) =\varepsilon \psi(x_-(0))+G(\varepsilon \varphi_x(x_{-}(0)))-\varepsilon \psi(x_+(0)) +G(\varepsilon \varphi_x(x_{+}(0))).
	\end{equation}

	\begin{figure}[http]
		\centering
		\begin{tikzpicture}[scale=1.0]
			\draw[->] (-3.7,0) -- (6,0) node[right] {$x$};
			\draw[->] (0,-0.5) -- (0,4) node[above] {$\tau$};

			\draw[thick, black, name path=minus] (3.8,0) node[below]{$x_0$} to[out=140, in=-80] (-0.5,4.0) node[left]{$x_-(\tau)$};

			\draw[thick, black!60!black, name path=plus] (-1.2,0) node[below]{$x_{+}(0)$} to[out=40, in=-115] (2.5,4.0) node[right]{$x_{+}(\tau)$};

			\path [name intersections={of=minus and plus, by=P}];
			\fill (P) circle (2pt) node[right, xshift=10pt] {$(t, x_-(t))$};

		\end{tikzpicture}
		\caption{Characteristics }
	\end{figure}
	Assume (i).
	Fix $t\in [0,T^*)$ arbitrarily, and consider the positive characteristic curve passing through $x_{-}(t)$ at time $t$:
	\begin{equation*}
		\begin{cases}
			\frac{dx_{+}}{d\tau }(\tau)=c(u_x(\tau,x_{+}(\tau))), \\[10pt]
			x_{+}(t)=x_{-}(t).
		\end{cases}
	\end{equation*}
	Then
	\begin{align*}
		x_{-}(t) & =x_0-\int_{0}^{t}c(u_x(\tau,x_{-}(\tau)))d\tau,      \\[5pt]
		x_{+}(t) & =x_{+}(0)+\int_{0}^{t}c(u_x(\tau,x_{+}(\tau)))d\tau.
	\end{align*}
	Since $x_{+}(t)=x_{-}(t)$, we have $x_{+}(0)  = \, x_0-\int_{0}^{t}\{c(u_x(\tau,x_{-}(\tau)))+c(u_x(\tau,x_{+}(\tau)))\}d\tau$.
	\vspace{5pt}
	From \eqref{Eq:c_near_1}, taking $\varepsilon>0$ sufficiently small allows us to have $x_{+}(0) < x_0-t$.
	Also, since $\psi(x), \varphi_x(x) \rightarrow 0$ as $|x| \rightarrow \infty$, for any $\delta>0$, there exists $R>0$ such that
	\begin{equation*}
		|\varphi_x(x)|,|\psi(x)| \leq \delta \hspace{8pt} \text{if} \hspace{8pt} x \leq -R.
	\end{equation*}
	Thus, for large time $t$ such that $t \geq R+x_0$,
	\begin{equation*}
		|\varepsilon \psi(x_+(0))-G(\varepsilon \varphi_x(x_+(0)))| \leq C\varepsilon \delta
	\end{equation*}
	by \eqref{Eq:G_to_u}.
	From \eqref{Eq:G_to_u},\eqref{Eq:G_expand}, we have
	\begin{equation*}
		|G(u_x(t,x_{-}(t)))|\geq C\varepsilon
	\end{equation*}
	Indeed, the Taylor expansion of $G(\varepsilon \varphi_x)$ yields for some $\xi \in (0,1)$
	\begin{equation*}
		G(\varepsilon \varphi_x)=\varepsilon \varphi_x + \frac{1}{2}c'(\xi \varepsilon \varphi_x)(\varepsilon \varphi_x)^2
	\end{equation*}
	and by taking $\delta>0$ such that $C\delta<|\psi(x_0)+\varphi_x(x_0)|-C'\varepsilon_0^{p-1}$, we obtain for $0<\varepsilon \leq \varepsilon_0$ that
	\begin{align*}
		|2G(u_x)| & \geq |\varepsilon \psi(x_0)+G(\varepsilon \varphi_x(x_0))|-C\varepsilon \delta       \\
		          & \geq |\psi(x_0)+\varphi_x(x_0)|\varepsilon-(C\delta+C'\varepsilon^{p-1})\varepsilon.
	\end{align*}
	Again, by \eqref{Eq:G_to_u},
	\begin{equation*}
		|u_x(t,x_-(t))| \geq \frac{2}{3}C\varepsilon.
	\end{equation*}
	Thus, taking $\varepsilon>0$ sufficiently small, \eqref{Eq:c_prime_lower1} holds for sufficiently large time $t$ by assumption \eqref{Eq:c assump2}.
	The case for assumption (ii) follows similarly, yielding \eqref{Eq:c_prime_lower2}.
\end{proof}

\vspace{5pt}

\begin{proof}[Proof of Theorem 2]
	We will only show the case $(\mathrm{i})$.
	Let $T\geq R+x_0$, and consider $t \geq T$. From \eqref{Eq:Riccati1}, on the characteristic $x_{-}(t)$ passing through $(0,x_0)$,
	\begin{equation*}
		F_1(t)=F_1(T)+\int_{T}^{t}\gamma(u_x(\tau))F_1(\tau)^2d\tau.
	\end{equation*}
	Furthermore, from \eqref{Eq:c_near_1} and \eqref{Eq:c_prime_lower1}, there exists a positive constant $\gamma_2>0$ such that for sufficiently small $\varepsilon>0$,
	\begin{equation*}
		\gamma_2\varepsilon^{p-2} \leq \gamma(u_x(t)).
	\end{equation*}
	By an argument similar to that in the proof of Lemma 5, Taylor expansion of $c(u_x)$ yields
	\begin{align*}
		 & F_1(0) =\sqrt{c(u_x(0,x_0))}(\varepsilon\psi_x(x_0)+\varepsilon c(u_x(0,x_0))\varphi_{xx}(x_0))\geq C\varepsilon , \\
		 & F_1(T)-F_1(0)  =\int_{0}^{T}\gamma(u_x(\tau))F_1(\tau)^2d\tau \geq 0.
	\end{align*}
	Hence,
	\begin{equation*}
		F_1(t)\geq C\varepsilon+\gamma_2\varepsilon^{p-2}\int_{T}^{t}F_1(\tau)^2d\tau.
	\end{equation*}
	Let $f(t)$ be the solution to the integral equation
	\begin{equation*}
		f(t)=C\varepsilon+\gamma_2\varepsilon^{p-2}\int_{T}^{t}f(\tau)^2d\tau.
	\end{equation*}
	Solving this integral equation, we have
	\begin{equation*}
		f(t)=\frac{C\varepsilon}{1-C\gamma_2\varepsilon^{p-1}(t-T)},
	\end{equation*}
	and we can easily check that \,$F_1(t) \geq f(t)$. Therefore,
	\begin{equation*}
		T^* \leq \frac{1}{C\gamma_2}\varepsilon^{-(p-1)}+T.
	\end{equation*}
	By taking $\varepsilon>0$ sufficiently small such that $T \leq \varepsilon^{-(p-1)}$,
	\begin{equation*}
		T^* \leq (\frac{1}{C\gamma_2}+1)\varepsilon^{-(p-1)}
	\end{equation*}
	The case $(\mathrm{ii})$ is similar if we consider \eqref{Eq:Riccati2}.
	This completes the proof.
\end{proof}

\section{The case where $c(\theta)$ is flat}
A similar exponential lifespan estimate was obtained by Hoshiga \cite{Hoshiga} for $2\times2$ quasilinear hyperbolic systems in the critical case.
Here we give a shorter proof based on the Riemann invariants of the present equation, under weaker decay assumptions on the initial data.
\begin{defin}[Gevrey Class]
	Let $\Omega$ be an open set in $\mathbb{R}$ and $s>1$. A function $f$ belongs to the Gevrey Class $G^s(\Omega)$ if there exist positive constants $C,A>0$ such that for all $m \in \mathbb{N}$,
	\begin{equation*}
		|f^{(m)}(\theta)| \leq CA^m(m!)^s, \hspace{10pt} \theta \in \Omega,
	\end{equation*}
\end{defin}
The advantage of this generalization is that the introduction of Gevrey class enables us to treat functions that are flat at \,$\theta=0$.

\begin{thm}
	Let $\varphi \in C_b^2(\mathbb{R})$ and $\psi \in C_b^1(\mathbb{R})$.
	If
	\begin{equation}\label{Eq:c assump flat}
		c \in G^s(U),\hspace{3pt}c(0)=1,\,c'(0)=\dots=c^{(m)}(0)=0, \hspace{10pt} \forall m \in \mathbb{N}
	\end{equation}
	is satisfied, then there exists $\varepsilon_0>0$ such that
	\begin{equation}\label{Eq:T^* Gevrey1}
		C_1e^{C_2\varepsilon^{-\frac{1}{s-1}}} \leq T^* \hspace{8pt}\text{if}\hspace{9pt}0<\varepsilon \leq \varepsilon_0,
	\end{equation}
	where $s>1$ and $U$ is a neighborhood of origin.

	Also, assuming the same conditions on the initial data as in Theorem 2, if
	\begin{equation}\label{Eq:c assump Gevrey}
		Ce^{-C|\theta|^{-\frac{1}{s-1}}} \leq c'(\theta) \hspace{8pt} \text{near} \hspace{5pt} \theta=0,
	\end{equation}
	then there exists $\varepsilon_0>0$ such that
	\begin{equation}\label{Eq:T^* Gevrey2}
		T^*\leq C_3e^{C_4\varepsilon^{-\frac{1}{s-1}}} \hspace{8pt}\text{if}\hspace{9pt}0<\varepsilon \leq \varepsilon_0.
	\end{equation}
\end{thm}
Examples of the function $c(\theta)$ satisfying the assumptions of Theorem 7 include
the following:
\begin{equation*}
	c(\theta)=1+\int_{0}^{\theta}g(\tau)d\tau
\end{equation*}
where
\begin{equation*}
	g(\theta) =
	\begin{cases}
		C_1e^{-C_2|\theta|^{-\frac{1}{s-1}}}  \hspace{9pt} \theta \not=0 \\
		0 \hspace{9pt} \theta=0
	\end{cases}
\end{equation*}
or
\begin{equation*}
	g(\theta) =
	\begin{cases}
		C_1|\theta|^{-\alpha}e^{-C_2|\theta|^{-\frac{1}{s-1}}}  \hspace{9pt} \theta \not=0 \\
		0 \hspace{9pt} \theta=0
	\end{cases}
\end{equation*}
with $\alpha \in \mathbb{R}$ or
\begin{equation*}
	g(\theta) =
	\begin{cases}
		C_1e^{-(C_2|\theta|^{-\frac{1}{s-1}}+C_3|\theta|^{-\frac{1}{s'-1}})}  \hspace{9pt} \theta \not=0 \\
		0 \hspace{9pt} \theta=0
	\end{cases}
\end{equation*}
with $s'>s $.
\begin{proof}
	The theorem can be proven similarly if we can show the upper bound estimate for $\|c'(u_x)\|_{L^\infty}$:
	\begin{equation}\label{Eq:c'Gevrey}
		\|c'(u_x)\|_{L^\infty} \leq C_1e^{-C_2\varepsilon^{-\frac{1}{s-1}}}.
	\end{equation}
	Since this immediately follows from the following Proposition 8, it suffices to show the following inequality to complete the proof similar to Theorem 1.
	By \eqref{Eq:c'Gevrey}, the integral inequalities
	\begin{equation*}
		|F_j(t)|\leq C\varepsilon +C_1e^{-C_2\varepsilon^{-\frac{1}{s-1}}}\int_{0}^{t}F_j(\tau)^2d\tau \hspace{10pt} (j=1,2)
	\end{equation*}
	follow on the characteristics, and \eqref{Eq:T^* Gevrey1} holds similarly to Theorem 1.
	Furthermore, by the same argument as in Lemma 5,
	\begin{align*}
		\text{under the assumption (i),} \,\,  & c'(u_x) \geq C_1e^{-C_2\varepsilon^{-\frac{1}{s-1}}} \hspace{7pt}\text{on $x_{-}(t;0,x_0)$} , \\
		\text{under the assumption (ii),} \,\, & c'(u_x) \geq C_1e^{-C_2\varepsilon^{-\frac{1}{s-1}}} \hspace{7pt}\text{on $x_{+}(t;0,x_0)$}
	\end{align*}
	hold from the assumption \eqref{Eq:c assump Gevrey}.
	Therefore, under the assumption ($\mathrm{i}$), the integral inequality
	\begin{equation*}
		F_1(t)\geq C\varepsilon +C_1e^{-C_2\varepsilon^{-\frac{1}{s-1}}}\int_{0}^{t}F_1(\tau)^2d\tau
	\end{equation*}
	follows on the characteristics $x_{-}(t;0,x_0)$, and \eqref{Eq:T^* Gevrey2} holds. Similarly for the assumption ($\mathrm{ii}$).
\end{proof}

\begin{prop}
	We assume that $c(\theta)$ satisfies \eqref{Eq:c assump flat}. Then
	\begin{equation*}
		|c'(\theta)| \leq Ce^{-C'|\theta|^{-\frac{1}{s-1}}} \quad \text{for all } \theta \in U.
	\end{equation*}
\end{prop}
\begin{proof}
	This inequality is a classical characterization of Gevrey flatness. It appears, for example, as Proposition 5 in Ramis \cite{Ramis78} (see also \cite{Rammis84}).
	For the convenience of the reader, we include a proof here. Our argument is relatively direct and does not rely on Stirling’s formula. By Taylor expansion $c'(\theta)$ up to order $m$,
	\begin{align*}
		|c'(\theta)| & \leq CA^{m+1}\frac{((m+1)!)^s}{m!}|\theta|^m \\[3pt]
		             & = CA^{m+1}(m+1)^s(m!)^{s-1}|\theta|^m        \\
		             & \leq CA(m+1)^s|Am^{s-1}\theta|^m.
	\end{align*}
	We take the largest $m \in \mathbb{N}$ such that for $\theta$:
	\begin{equation*}
		m \leq \left(\frac{1}{2eA|\theta|}\right)^{\frac{1}{s-1}}.
	\end{equation*}
	Note that for this $m$, the following inequalities hold:
	\begin{equation*}
		Am^{s-1}|\theta| \leq \frac{1}{2e} \hspace{8pt} \text{and} \hspace{8pt} m > \left(\frac{1}{2eA|\theta|}\right)^{\frac{1}{s-1}}-1.
	\end{equation*}
	Therefore, noting that $\frac{(m+1)^s}{2^m} \leq C$,

	\begin{align*}
		|c'(\theta)| & \leq C_1A(m+1)^s \left(\frac{1}{2e}\right)^m \\
		             & \leq C_2Ae^{-m}                              \\
		             & \leq C_3Ae^{-C'|\theta|^{-\frac{1}{s-1}}}.
	\end{align*}

\end{proof}

\section*{Acknowledgements}
Y. Sugiyama was partially supported by Grants-in-Aid for Scientific Research (C) (No. 23K03169).
Also, the authors would like to thank Professor Hideaki Sunagawa for bringing the references \cite{Hoshiga,LZK} to their attention.

\vspace{30pt}

Yuusuke Sugiyama

Department of Mathematics,

Tokyo University of Science,

Kagurazaka 1-3, Shinjuku-ku,

Tokyo 162-8601, Japan

E-mail: sugiyama.y@rs.tus.ac.jp

\vspace{17pt}

Taro Yamanoi

Department of Mathematics,

Faculty of Science,

Tokyo University of Science,

Kagurazaka 1-3, Shinjuku-ku,

Tokyo 162-8601, Japan

E-mail: 1126709@ed.tus.ac.jp


\begin{thebibliography}{9}
	\bibitem{CPZ}
	G. Chen, R. Pan, S. Zhu, Singularity formation for the compressible Euler equations, SIAM J. Math. Anal. 49 (2017), no.4, 2591–2614.

	\bibitem{Courant}
	R. Courant, P. Lax, On nonlinear partial differential equations with two independent variables, Comm. Pure Appl. Math. 2 (1949), 255--273.

	\bibitem{Douglis}
	A. Douglis, Some existence theorems for hyperbolic systems of partial differential equations in two independent variables. Comm. Pure Appl. Math. 5 (1952), 119-154.

	\bibitem{Philip}
	P. Hartman and A. Wintner, On hyperbolic partial differential equations. Amer. J. Math. 74 (1952), 834--864.

	\bibitem{Haruyama}
	Y. Haruyama and H. Takamura, Blow-up of classical solutions of quasilinear wave equations in one space dimension, Nonlinear Anal. Real World Appl. 81 (2025), 104212.


	\bibitem{Hoshiga}
	A. Hoshiga, The lifespan of solutions to quasilinear hyperbolic systems in the critical case, Funkcial. Ekvac., 41 (1998), pp. 167-188.

	\bibitem{John}
	F. John, Formation of singularities in one-dimensional nonlinear wave propagation, Comm. Pure Appl. Math. 27 (1974), 377-405.

	\bibitem{Lax57}
	P. D. Lax, Hyperbolic systems of conservation laws. II. Comm. Pure Appl. Math. 10 (1957), 537-566.

	\bibitem{Lax64}
	P. D. Lax, Development of Singularities of solutions of nonlinear hyperbolic partial differential equations, J. Math. Phys. 5 (1964), 611--613.

	\bibitem{LeVeque}
	R. J. LeVeque, Finite Volume Methods for Hyperbolic Problems, Cambridge University Press,
	Cambridge, UK (2002).

	\bibitem{Manfrin}
	R. Manfrin, On the breakdown of solutions to $N \times N$ quasilinear hyperbolic systems, Nonlinear Anal. 52 (2003), no. 1, 143-172.

	\bibitem{Ramis78}
	J.-P. Ramis, Devissage Gevrey, Soc. Math. de France Asterisque 59-60 (1978), 173-204.

	\bibitem{Rammis84}
	J.-P. Ramis, Théorèmes d’indices Gevrey pour les équations différentielles ordinaires, Memoirs of the American Mathematical Society 48, no. 296 (1984).

	\bibitem{Sugiyama16}
	Y. Sugiyama, Degeneracy in finite time of 1D quasilinear wave equations, SIAM J. Math. Anal. 48 (2016), no.2, 847-860.

	\bibitem{LZK}
	Li, T. T., Zhou, Y., Kong, D. X., Global classical solutions for general quasilinear hyperbolic systems with decay initial data, Nonlinear Analysis, TMA, 28(1997), 1299-1332.

\end{thebibliography}
\end{document}